\documentclass[11pt,a4paper]{article}
\usepackage{amssymb}
\usepackage{amsmath}
\usepackage{latexsym}
\usepackage{amsthm}
\usepackage{dsfont}

\newtheorem{thm}{Theorem}

\theoremstyle{definition}

\usepackage{xpatch}
\xpatchcmd{\proof}{\itshape}{\normalfont\proofnameformat}{}{}
\newcommand{\proofnameformat}{}

\begin{document}

\renewcommand{\proofnameformat}{\bfseries}

\begin{center}
{\Large\textbf{Remarks on sums of reciprocals of fractional parts}}

\vspace{10mm}

\textbf{Bence Borda}

{\footnotesize Graz University of Technology

Steyrergasse 30, 8010 Graz, Austria

Email: \texttt{borda@math.tugraz.at}}

\vspace{5mm}

{\footnotesize \textbf{Keywords:} Diophantine approximation, continued fraction,\\ small fractional parts, Diophantine sum, metric number theory}

{\footnotesize \textbf{Mathematics Subject Classification (2020):} 11J54, 11J71, 11J83}
\end{center}

\vspace{5mm}

\begin{abstract}
The Diophantine sums $\sum_{n=1}^N \| n \alpha \|^{-1}$ and $\sum_{n=1}^N n^{-1} \| n \alpha \|^{-1}$ appear in many different areas including the ergodic theory of circle rotations, lattice point counting and random walks, often in connection with Fourier analytic methods. Beresnevich, Haynes and Velani gave estimates for these and related sums in terms of the Diophantine approximation properties of $\alpha$ that are sharp up to a constant factor. In the present paper, we remove the constant factor gap between the upper and the lower estimates, and thus find the precise asymptotics for a wide class of irrationals. Our methods apply to sums with the fractional part instead of the distance from the nearest integer function, and to sums involving shifts $\| n \alpha + \beta \|$ as well. We also comment on a higher dimensional generalization of these sums.
\end{abstract}

\section{Introduction}

The subject of this paper is the asymptotic behavior of the Diophantine sums
\begin{equation}\label{diophantinesums}
\sum_{n=1}^N \frac{1}{\| n \alpha \|} \qquad \textrm{and} \qquad \sum_{n=1}^N \frac{1}{n \| n \alpha \|}
\end{equation}
for various irrational $\alpha$, where $\| \cdot \|$ denotes the distance from the nearest integer function. These sums appear in many different fields such as uniform distribution theory \cite{HO,KN}, multiplicative Diophantine approximation \cite{BHV,CT}, lattice point counting in polygons \cite{BE2,BO1,HL1,HL2,SK}, dynamical systems \cite{BO2,DF1,DF2,DS} and random walks \cite{BB1,BB2,BW,WE}. We refer to \cite{BEL,BHV} for a comprehensive survey.

The behavior of the more general sum $\sum_{n=1}^N n^{-p} \| n \alpha \|^{-q}$ is highly sensitive to the value of the exponents $p,q \ge 0$ \cite{KR}. Sharp estimates in the case $p=q=2$ were given in \cite{BE2}, and in the case $p=q>1$ in \cite{BO3}. The case $p>q$ leads to convergent series for certain irrationals \cite{CM}. For higher dimensional generalizations of the sums in \eqref{diophantinesums}, see \cite{BE1,FR1,FR2,LV}.

Hardy and Littlewood \cite{HL1,HL2,HL3}, Haber and Osgood \cite{HO}, Kruse \cite{KR} and more recently Beresnevich, Haynes and Velani \cite{BHV} gave estimates for the sums in \eqref{diophantinesums} in terms of the Diophantine approximation properties of $\alpha$ that are sharp up to a constant factor. The main goal of the present paper is to remove the constant factor gap between the upper and the lower estimates, thereby establishing the precise asymptotics.

Recall that an irrational $\alpha$ is called badly approximable if $\inf_{n \in \mathbb{N}} n \| n \alpha \| >0$. The best known estimates for such an $\alpha$ are\footnote{For the sake of readability, with a slight abuse of notation in all error terms $\log x := \log \max \{ e,x \}$.} \cite{BHV}
\[ N \log N \le \sum_{n=1}^N \frac{1}{\| n \alpha \|} \ll N \log N \qquad \textrm{and} \qquad \frac{1}{2} (\log N)^2 \le \sum_{n=1}^N \frac{1}{n \| n \alpha \|} \le 33 (\log N)^2 +O(\log N) . \]
Our first result improves these.
\begin{thm}\label{badlyapproximabletheorem} Let $\alpha$ be a badly approximable irrational. For any $N \ge 1$,
\[ \sum_{n=1}^N \frac{1}{\| n \alpha \|} = 2N \log N +O(N) \qquad \textrm{and} \qquad \sum_{n=1}^N \frac{1}{n \| n \alpha \|} = (\log N)^2 +O(\log N) \]
with implied constants depending only on $\alpha$.
\end{thm}

The only previously known precise asymptotic result (without a constant factor gap between the upper and the lower estimates) for the sums in \eqref{diophantinesums} appeared in an obscure paper of Erd\H{o}s from 1948 \cite{ER}, in which he showed that
\begin{equation}\label{erdos}
\sum_{\substack{n=1 \\ \| n \alpha \| \ge 1/(2N)}}^N \frac{1}{\| n \alpha \|} \sim 2N \log N \qquad \textrm{and} \qquad \sum_{n=1}^N \frac{1}{n \| n \alpha \|} \sim (\log N)^2 \qquad \textrm{for a.e. } \alpha .
\end{equation}
This result seems to have escaped the attention of several later authors who subsequently proved weaker estimates in the metric setting. We learned about \eqref{erdos} from the recent historical survey \cite{BEL}, and it served as the starting point of our investigations. Our next result improves \eqref{erdos} by finding the precise order of the error term.
\begin{thm}\label{aetheorem} Let $c>0$ be an arbitrary constant, and let $\varphi$ be a positive nondecreasing function on $(0,\infty)$. If $\sum_{k=1}^{\infty} 1/\varphi(k)<\infty$, then for a.e.\ $\alpha$,
\[ \begin{split} \sum_{\substack{n=1 \\ \| n \alpha \| \ge c/N}}^N \frac{1}{\| n \alpha \|} &= 2N \log N + O( N \varphi(\log N)^{1/2}), \\ \sum_{n=1}^N \frac{1}{n \| n \alpha \|} &= (\log N)^2 + O(\varphi(\log N) + \log N \log \log N) \end{split} \]
with implied constants depending only on $c$, $\varphi$ and $\alpha$. If $\sum_{k=1}^{\infty} 1/\varphi (k)=\infty$, then for a.e.\ $\alpha$ the sets
\[ \begin{split} \Bigg\{ N \in \mathbb{N} \, : \, \sum_{\substack{n=1 \\ \| n \alpha \| \ge c/N}}^N \frac{1}{\| n \alpha \|} &\ge 2N \log N + N \varphi (\log N)^{1/2} \Bigg\} , \\ \Bigg\{ N \in \mathbb{N} \, : \, \sum_{n=1}^N \frac{1}{n \| n \alpha \|} &\ge (\log N)^2 + \varphi (\log N) \Bigg\} \end{split} \]
have upper asymptotic density 1.
\end{thm}
\noindent Recall that the upper asymptotic density of a set $A \subseteq \mathbb{N}$ is defined as $\limsup_{N \to \infty} |A \cap [1,N]|/N$. In particular, for a.e.\ $\alpha$ we have
\[ \begin{split} \sum_{\substack{n=1 \\ \| n \alpha \| \ge c/N}}^N \frac{1}{\| n \alpha \|} &= 2N \log N +O(N (\log N)^{1/2} (\log \log N)^{1/2+\varepsilon}), \\ \sum_{n=1}^N \frac{1}{n \| n \alpha \|} &= (\log N)^2 +O(\log N (\log \log N)^{1+\varepsilon}) \end{split} \]
with any $\varepsilon >0$, but these fail with $\varepsilon =0$. Note that the cutoff $\| n \alpha \| \ge c/N$ in the first sum is necessary, as for a.e.\ $\alpha$ we have $\| n \alpha \| < (n \log n \log \log n)^{-1}$ for infinitely many $n \in \mathbb{N}$.

Our main results, Theorems \ref{firstsumtheorem} and \ref{secondsumtheorem} on the sums in \eqref{diophantinesums} with a general irrational $\alpha$ are presented in Section \ref{mainsection}. We discuss closely related sums involving fractional parts and shifts $\| n \alpha +\beta \|$, and comment on a higher dimensional generalization in Section \ref{relatedsection}.

\section{Main estimates}\label{mainsection}

For the rest of the paper, $\alpha$ is an irrational number with continued fraction $\alpha = [a_0 ; a_1, a_2, \ldots]$ and convergents $p_k/q_k=[a_0 ; a_1, a_2, \ldots, a_k]$. Set $s_K=\sum_{k=1}^K a_k$. We refer to \cite{KH} for a general introduction to continued fractions.

\subsection{General irrationals}

In this section, we prove our main estimates for the sums in \eqref{diophantinesums} with a general irrational $\alpha$.
\begin{thm}\label{firstsumtheorem} Let $c>0$ be an arbitrary constant. For any $K \ge 0$ and $q_K \le N<q_{K+1}$,
\[ \sum_{\substack{n=1 \\ \| n \alpha \| \ge c/N}}^N \frac{1}{\| n \alpha \|} = 2N \log N +O\left( (a_{K+1}^{1/2} + \log s_{K+1}) N \right) . \]
If in addition $4 (c a_{K+1})^{1/2} q_K \le N$, then
\[ \sum_{\substack{n=1 \\ \| n \alpha \| \ge c/N}}^N \frac{1}{\| n \alpha \|} \ge 2N \log N + q_{K+1} -O((\log s_{K+1})N) . \]
The implied constants depend only on $c$.
\end{thm}

\begin{proof} If $N \le 4c$, then
\[ 0 \le \sum_{\substack{n=1 \\ \| n \alpha \| \ge c/N}}^N \frac{1}{\| n \alpha \|} \le \frac{N^2}{c} \ll 1, \]
and the claims trivially follow. We may thus assume that $N>4c$.

Summing the identity $\| n \alpha \|^{-1} = \int_0^{\infty} t^{-2} \mathds{1}_{\{ \| n \alpha \| \le t \}} \, \mathrm{d}t$ over all $1 \le n \le N$ such that $\| n \alpha \| \ge c/N$ yields
\begin{equation}\label{integral}
\sum_{\substack{n=1 \\ \| n \alpha \| \ge c/N}}^N \frac{1}{\| n \alpha \|} = \int_{c/N}^{\infty} \frac{1}{t^2} \left| \left\{ 1 \le n \le N \, : \, c/N \le \| n \alpha \| \le t \right\} \right| \, \mathrm{d}t .
\end{equation}
Let $R=\frac{1}{8} (s_K+N/q_K)+c$, and note that $c/N \le R/N \le 1/2$. The latter inequality follows from the assumption $c<N/4$ and the general estimate $s_K \le q_K$, which can be easily seen e.g.\ by induction on $K$. We will estimate the integral in \eqref{integral} on the intervals $[c/N,R/N]$, $[R/N,1/2]$ and $[1/2,\infty)$ separately.

The integral on $[1/2,\infty)$ is negligible:
\begin{equation}\label{integral1/2infty}
\int_{1/2}^{\infty} \frac{1}{t^2} \left| \left\{ 1 \le n \le N \, : \, c/N \le \| n \alpha \| \le t \right\} \right| \, \mathrm{d}t \le \int_{1/2}^{\infty} \frac{N}{t^2} \, \mathrm{d}t = 2N.
\end{equation}
Consider now the integral on $[R/N,1/2]$. Let $D_N(\alpha)$ denote the discrepancy of the point set $\{ n \alpha \}$, $1 \le n \le N$. That is,
\[ D_N(\alpha) = \sup_{I \subseteq [0,1]} \left| \sum_{n=1}^N \mathds{1}_I (\{ n \alpha \}) - \lambda (I) N \right| , \]
where the supremum is over all intervals $I \subseteq [0,1]$, and $\lambda$ is the Lebesgue measure. In particular,
\[ \left| \left\{ 1 \le n \le N \, : \, c/N \le \| n \alpha \| \le t \right\} \right| = 2 \left( t-\frac{c}{N} \right) N +O(D_N(\alpha)) \quad \textrm{uniformly in } t \in [R/N,1/2] . \]
A classical discrepancy estimate \cite[p.\ 126]{KN} states that $D_N(\alpha) \le 2 (s_K + N/q_K) \ll R$. Therefore
\begin{equation}\label{integralR/N1/2}
\begin{split}  \int_{R/N}^{1/2} \frac{1}{t^2} \left| \left\{ 1 \le n \le N \, : \, c/N \le \| n \alpha \| \le t \right\} \right| \, \mathrm{d}t &=  \int_{R/N}^{1/2} \frac{2Nt-2c+O(D_N(\alpha))}{t^2} \, \mathrm{d}t \\ &=2N \log N +O (N \log R) . \end{split}
\end{equation}
Finally, consider the integral on $[c/N,R/N]$. We will show that
\begin{equation}\label{integralc/NR/N}
\int_{c/N}^{R/N} \frac{1}{t^2} \left| \left\{ 1 \le n \le N \, : \, \| n \alpha \| \le t \right\} \right| \, \mathrm{d}t \ll (a_{K+1}^{1/2} + \log R) N .
\end{equation}
For any $1 \le n < n' \le N$ we have $1 \le n'-n <q_{K+1}$, hence by the best rational approximation property of continued fraction convergents, $\| n' \alpha - n \alpha \| \ge \| q_K \alpha \| \ge 1/(2q_{K+1})$. The pigeonhole principle thus shows that
\[ \left| \left\{ 1 \le n \le N \, : \, \| n \alpha \| \le t \right\} \right| \le 4 q_{K+1} t +1, \]
therefore
\begin{equation}\label{pigeonhole}
\int_{c/N}^{R/N} \frac{1}{t^2} \left| \left\{ 1 \le n \le N \, : \, \| n \alpha \| \le t \right\} \right| \, \mathrm{d}t \ll q_{K+1} \log R .
\end{equation}
The previous formula implies \eqref{integralc/NR/N} whenever $N \gg q_{K+1}$.

It remains to prove \eqref{integralc/NR/N} when, say, $N<q_{K+1}/4$. Let $A = \{ 1 \le n \le N \, : \, q_K \mid n \}$ and $B = \{ 1 \le n \le N \, : \, q_K \nmid n \}$. Consider an element $n=jq_K \in A$ with some integer $1 \le j \le N/q_K$, and observe that $\| jq_K \alpha \| \le t$ implies that $j \le 2q_{K+1}t$. In particular,
\[ \left| \left\{ n \in A \, : \, \| n \alpha \| \le t \right\} \right| \le \min \left\{ \frac{N}{q_K}, 2q_{K+1} t \right\} . \]
If $N \le (q_K q_{K+1})^{1/2}$, then
\[ \int_{c/N}^{R/N} \frac{1}{t^2} \left| \left\{ n \in A \, : \, \| n \alpha \| \le t \right\} \right| \, \mathrm{d}t \le \int_{c/N}^{R/N} \frac{N/q_K}{t^2} \, \mathrm{d}t \le \frac{N^2}{c q_K} \ll a_{K+1}^{1/2} N . \]
If $N> (q_K q_{K+1})^{1/2}$, then again
\[ \begin{split} \int_{c/N}^{R/N} \frac{1}{t^2} \left| \left\{ n \in A \, : \, \| n \alpha \| \le t \right\} \right| \, \mathrm{d}t &\le \int_{c/N}^{cN/(q_K q_{K+1})} \frac{2q_{K+1}t}{t^2} \, \mathrm{d}t + \int_{cN/(q_K q_{K+1})}^{\infty} \frac{N/q_K}{t^2} \, \mathrm{d}t \\ &\ll q_{K+1} \log \frac{N^2}{q_K q_{K+1}} + q_{K+1} \\ &\ll a_{K+1}^{1/2} N . \end{split} \]

Consider now $B_j = \{ n \in B \, : \, jq_K < n < (j+1)q_K \}$ with some integer $0 \le j < \lceil N/q_K \rceil$. Observe that $|n \alpha - np_K/q_K| \le n/(q_K q_{K+1}) < 1/q_K$, and that $n p_K \pmod{q_K}$ attains nonzero residue classes. Therefore the set $\{ n \alpha \}$, $n \in B_j$ coincides with $\{ \ell/q_K \, : \, 1 \le \ell \le q_K-1 \}$ (or a subset thereof in case $j=\lceil N/q_K \rceil-1$) up to an error of $1/q_K$.

Since $\mathrm{sgn}(\alpha - p_K/q_K)=(-1)^K$, for all $n \in B_j$ with $n p_K \not\equiv (-1)^{K-1} \pmod{q_K}$, we have $\| n \alpha \| \ge 1/q_K$. The identity $q_K p_{K-1} - q_{K-1} p_K = (-1)^K$ shows that the unique $n \in B_j$ with $n p_K \equiv (-1)^{K-1} \pmod{q_K}$ is $n=jq_K +q_{K-1}$. The assumption $N<q_{K+1}/4$ ensures that
\[ \left\| (jq_K +q_{K-1}) \alpha  \right\| = \| q_{K-1} \alpha \| - j \| q_K \alpha \| \ge \frac{1}{2q_K} - \frac{N}{q_K} \cdot \frac{1}{q_{K+1}} > \frac{1}{4q_K} . \]
In particular, $\| n \alpha \| \ge 1/(4q_K)$ for all $n \in B$. Since $\{ n \alpha \}$, $n \in B_j$ is well approximated by (a subset of) $\{ \ell /q_K \, : \, 1 \le \ell \le q_K-1 \}$, we have
\[ \left| \left\{ n \in B_j \, : \, \| n \alpha \| \le t \right\} \right| \ll q_K t \quad \textrm{uniformly in } 1/(4q_K) \le t \le 1/2, \]
and by summing over $0\le j < \lceil N/q_K \rceil$,
\[ \left| \left\{ n \in B \, : \, \| n \alpha \| \le t \right\} \right| \ll N t \quad \textrm{uniformly in } 1/(4q_K) \le t \le 1/2 . \]
Therefore
\[ \int_{c/N}^{R/N} \frac{1}{t^2} \left| \left\{ n \in B \, : \, \| n \alpha \| \le t \right\} \right| \, \mathrm{d}t \ll \int_{1/(4q_K)}^{R/N} \frac{Nt}{t^2} \, \mathrm{d}t \ll N \log R . \]
Adding our estimates for $n \in A$ and $n \in B$ leads to \eqref{integralc/NR/N}. Formulas \eqref{integral}--\eqref{integralc/NR/N} yield
\[ \sum_{\substack{n=1 \\ \| n \alpha \| \ge c/N}}^N \frac{1}{\| n \alpha \|} = 2N \log N + O\left( (a_{K+1}^{1/2} + \log R) N \right) = 2N \log N +O \left( (a_{K+1}^{1/2} + \log s_{K+1})N \right) , \]
as claimed.

Suppose now that $4 (c a_{K+1})^{1/2} q_K \le N$, and let us prove the lower bound. We may assume that $a_{K+1}$ is large in terms of $c$. In particular, $R/N > 8c/N$, and formulas \eqref{integral}--\eqref{integralR/N1/2} yield
\[ \sum_{\substack{n=1 \\ \| n \alpha \| \ge c/N}}^N \frac{1}{\| n \alpha \|} \ge 2N \log N + \int_{c/N}^{8c/N} \frac{1}{t^2} \left| \left\{ 1 \le n \le N \, : \, c/N \le \| n \alpha \| \le t \right\} \right| \, \mathrm{d}t - O(N \log R) . \]
Let $c/N \le t \le 8c/N$. Observe that $(q_K+q_{K+1})c/N \le j \le tq_{K+1}$ implies both $c/N \le \| j q_K \alpha \| \le t$ and $j \le N/q_K$, the latter by the assumption $4 (c a_{K+1})^{1/2} q_K \le N$. Therefore
\[ \left| \left\{ 1 \le n \le N \, : \, c/N \le \| n \alpha \| \le t \right\} \right| \ge tq_{K+1} - (q_K+q_{K+1}) c/N -1, \]
and so
\[ \int_{c/N}^{8c/N} \frac{1}{t^2} \left| \left\{ 1 \le n \le N \, : \, c/N \le \| n \alpha \| \le t \right\} \right| \, \mathrm{d}t \ge (\log 8) q_{K+1} - \frac{7}{8} (q_K +q_{K+1}) - \frac{N}{c} \ge q_{K+1}-O(N) . \]
Hence
\[ \sum_{\substack{n=1 \\ \| n \alpha \| \ge c/N}}^N \frac{1}{\| n \alpha \|} \ge 2N \log N + q_{K+1} - O(N \log R) = 2N \log N + q_{K+1} - O((\log s_{K+1}) N) , \]
as claimed.
\end{proof}

\begin{thm}\label{secondsumtheorem} For any $K \ge 0$ and $q_K \le N < q_{K+1}$,
\[ \sum_{n=1}^N \frac{1}{n \| n \alpha \|} = (\log N)^2 + \frac{\pi^2}{6} s_K + a_{K+1} \sum_{1 \le j \le N/q_K} \frac{1}{j^2} + O \left( \sum_{k=1}^{K+1} a_k^{1/2} \log a_k + (\log s_{K+1}) \log N \right) \]
with a universal implied constant.
\end{thm}

\begin{proof} We estimate the contribution of the terms with $\| n \alpha \| < 1/(2n)$ and $\| n \alpha \| \ge 1/(2n)$ separately.

Let $0 \le k \le K$, and consider the integers $q_k \le n < q_{k+1}$. If $\| n \alpha \| < 1/(2n)$, then by Legendre's theorem \cite[p.\ 30]{KH} we have $q_k \mid n$. A given multiple $n=jq_k$ satisfies $\| j q_k \alpha \| < 1/(2jq_k)$ if and only if $j < (2 q_k \| q_k \alpha \|)^{-1/2}$. Therefore
\[ \sum_{\substack{q_k \le n < q_{k+1} \\ \| n \alpha \| < 1/(2n)}} \frac{1}{n \| n \alpha \|} = \sum_{1 \le j < (2q_k \| q_k \alpha \|)^{-1/2}} \frac{1}{j q_k \| j q_k \alpha \|}, \]
consequently
\begin{equation}\label{nalpha<1/2n}
\begin{split} \sum_{\substack{n=1 \\ \| n \alpha \| < 1/(2n)}}^N \frac{1}{n \| n \alpha \|} &= \sum_{k=0}^{K-1} \sum_{1 \le j < (2q_k \| q_k \alpha \|)^{-1/2}} \frac{1}{j q_k \| j q_k \alpha \|} + \sum_{1 \le j \le \min \{ (2q_K \| q_K \alpha \|)^{-1/2}, N/q_K \}}  \frac{1}{j q_K \| j q_K \alpha \|} \\ &= \sum_{k=0}^{K-1} \frac{1}{q_k \| q_k \alpha \|} \left( \frac{\pi^2}{6} + O(a_{k+1}^{-1/2}) \right) + \frac{1}{q_K \| q_K \alpha \|} \left( \sum_{1 \le j \le N/q_K} \frac{1}{j^2} +O(a_{K+1}^{-1/2}) \right) \\ &= \frac{\pi^2}{6} \sum_{k=1}^K a_k + a_{K+1} \sum_{1 \le j \le N/q_K} \frac{1}{j^2} +O \left( \sum_{k=1}^{K+1} a_k^{1/2} \right) . \end{split}
\end{equation}

Letting
\[ T_N = \sum_{\substack{n=1 \\ \| n \alpha \| \ge 1/(2N)}}^N \frac{1}{\| n \alpha \|} , \]
we have
\[ \begin{split} \frac{1}{n \| n \alpha \|} \mathds{1}_{\{ \| n \alpha \| \ge 1/(2n) \}} &= \frac{1}{n} \Bigg( \sum_{\substack{\ell =1 \\ \| \ell \alpha \| \ge 1/(2n)}}^n \frac{1}{\| \ell \alpha \|} - \sum_{\substack{\ell =1 \\ \| \ell \alpha \| \ge 1/(2n)}}^{n-1} \frac{1}{\| \ell \alpha \|} \Bigg) \\ &= \frac{1}{n} \Bigg( T_n - T_{n-1} - \sum_{\substack{\ell =1 \\ 1/(2n) \le \| \ell \alpha \| < 1/(2(n-1))}}^{n-1} \frac{1}{\| \ell \alpha \|} \Bigg) \\ &= \frac{T_n-T_{n-1}}{n} +O \left( \sum_{\ell=1}^{n-1} \mathds{1}_{\{ 1/(2n) \le \| \ell \alpha \| < 1/(2(n-1)) \}} \right) . \end{split} \]
Summation by parts resp.\ switching the order of summation leads to
\[ \sum_{\substack{n=1 \\ \| n \alpha \| \ge 1/(2n)}}^N \frac{1}{n \| n \alpha \|} = \sum_{n=1}^{N-1} \frac{T_n}{n(n+1)} + \frac{T_N}{N} + O \left( \sum_{\ell=1}^{N-1} \mathds{1}_{\{ 1/(2N) \le \| \ell \alpha \| < 1/(2\ell) \}} \right) . \]
As above, Legendre's theorem shows that $\sum_{q_k \le \ell <q_{k+1}} \mathds{1}_{\{ \| \ell \alpha \| <1/(2\ell) \}} \ll a_{k+1}^{1/2}$, hence the error term in the previous formula is $\sum_{\ell=1}^{N-1} \mathds{1}_{\{ 1/(2N) \le \| \ell \alpha \| < 1/(2\ell) \}} \ll \sum_{k=1}^{K+1} a_k^{1/2}$. An application of Theorem \ref{firstsumtheorem} with $c=1/2$ gives
\[ \begin{split} \sum_{n=1}^{N-1} \frac{T_n}{n(n+1)} &= \sum_{n=1}^{N-1} \frac{2 \log n}{n+1} + O \left( \sum_{k=0}^K \sum_{q_k \le n < q_{k+1}} \frac{a_{k+1}^{1/2}}{n} + \sum_{n=1}^{N-1} \frac{\log s_{K+1}}{n} \right) \\ &= (\log N)^2 + O \left( \sum_{k=1}^{K+1} a_k^{1/2} \log a_k + (\log s_{K+1}) \log N \right) , \end{split} \]
and $T_N/N \ll \log N + a_{K+1}^{1/2} + \log s_{K+1}$. Therefore
\[ \sum_{\substack{n=1 \\ \| n \alpha \| \ge 1/(2n)}}^N \frac{1}{n \| n \alpha \|} = (\log N)^2 + O \left( \sum_{k=1}^{K+1} a_k^{1/2} \log a_k + (\log s_{K+1}) \log N \right) , \]
which together with \eqref{nalpha<1/2n} proves the claim.
\end{proof}

\subsection{Corollaries}

Theorems \ref{firstsumtheorem} and \ref{secondsumtheorem} establish the asymptotics of the sums in \eqref{diophantinesums} for a large class of irrational $\alpha$. For instance, we immediately obtain
\[ \begin{split} \sum_{\substack{n=1 \\ \| n \alpha \| \ge c/N}}^N \frac{1}{\| n \alpha \|} &\sim 2N \log N \quad \textrm{if } a_k = o(k^2), \\ \sum_{n=1}^N \frac{1}{n \| n \alpha \|} &\sim (\log N)^2 \quad \textrm{if } s_k=o(k^2) . \end{split} \]
As a further example, consider Euler's number $e=[2;1,2,1,\ldots, 1,2n,1,\ldots]$. The convergent denominators grow at the rate $\log q_k = (k/3) \log k +O(k)$, and $s_K=K^2/9+O(K)$. Theorems \ref{firstsumtheorem} and \ref{secondsumtheorem} thus give
\[ \begin{split} \sum_{\substack{n=1 \\ \| n \alpha \| \ge c/N}}^N \frac{1}{\| ne \|} &= 2N \log N + O \left( \frac{N (\log N)^{1/2}}{(\log \log N)^{1/2}} \right) , \\ \sum_{n=1}^N \frac{1}{n \| ne \|} &= (\log N)^2 + \frac{\pi^2}{6} \cdot \frac{(\log N)^2}{(\log \log N)^2} + O \left( \frac{(\log N)^{3/2}}{(\log \log N)^{1/2}} \right) . \end{split} \]

We will need certain basic facts from the metric theory of continued fractions in order to deduce the a.e.\ asymptotics from Theorems \ref{firstsumtheorem} and \ref{secondsumtheorem}. Khinchin and L\'evy showed that $\log q_k \sim \frac{\pi^2}{12 \log 2} k$ for a.e.\ $\alpha$, whereas Borel and Bernstein proved that given a positive function $\varphi$, for a.e.\ $\alpha$ we have $a_k \ge \varphi (k)$ for infinitely many $k \in \mathbb{N}$ if and only if $\sum_{k=1}^{\infty} 1/\varphi (k) =\infty$. A theorem of Diamond and Vaaler \cite{DV} on trimmed sums of partial quotients states that
\begin{equation}\label{diamondvaaler}
\lim_{K \to \infty} \frac{\sum_{k=1}^K a_k - \max_{1 \le k \le K} a_k}{K \log K} = \frac{1}{\log 2} \qquad \textrm{for a.e. } \alpha .
\end{equation}
We refer to the monograph \cite{IK} for the proof of all these results and for more context.

\begin{proof}[Proof of Theorem \ref{aetheorem}] Assume first that $\sum_{k=1}^{\infty} 1/\varphi(k)<\infty$. By the Khinchin--L\'evy and Borel--Bernstein theorems mentioned above, we have $a_{K+1} \le \varphi (K/100) \le \varphi (\log N)$ for all but finitely many $K$, and also $\log s_{K+1} \ll \log K \ll \log \log N$. Theorem \ref{firstsumtheorem} thus yields
\[ \sum_{\substack{n=1 \\ \| n \alpha \| \ge c/N}}^N \frac{1}{\| n \alpha \|} = 2N \log N + O(N \varphi (\log N)^{1/2} + N \log \log N) . \]
Since $\varphi$ is nondecreasing and $\sum_{k=1}^{\infty} 1/\varphi (k)<\infty$, we have $k/\varphi (k) \to 0$. In particular, the $N \log \log N$ error term in the previous formula is negligible compared to $N \varphi(\log N)^{1/2}$, as claimed.

By the Diamond--Vaaler theorem \eqref{diamondvaaler},
\[ s_{K+1} \ll \max_{1 \le k \le K+1} a_k + K \log K \ll \varphi (K/100) + K \log K \ll \varphi (\log N) + \log N \log \log N . \]
Theorem \ref{secondsumtheorem} thus yields
\[ \sum_{n=1}^N \frac{1}{n \| n \alpha \|} = (\log N)^2 + O \left( \varphi(\log N) + \log N \log \log N \right) , \]
as claimed.

Assume next that $\sum_{k=1}^{\infty} 1/\varphi(k) = \infty$. Fix a small $\varepsilon>0$. The function $\varphi^*(x)=4 \varepsilon^{-2} \varphi (100x)+x$ is also positive and nondecreasing, and satisfies $\sum_{k=1}^{\infty} 1/\varphi^*(k)=\infty$. By the Borel--Bernstein theorem, for a.e.\ $\alpha$ we have $a_{K+1} \ge \varphi^*(K) \ge K$ for infinitely many $K$. Theorem \ref{firstsumtheorem} gives that for infinitely many $K$ and any $4(c a_{K+1})^{1/2} q_K \le N \le \varepsilon^{-1} a_{K+1}^{1/2} q_K$,
\[ \begin{split} \sum_{\substack{n=1 \\ \| n \alpha \| \ge c/N}}^N \frac{1}{\| n \alpha \|} &\ge 2N \log N + q_{K+1} - O((\log s_{K+1} )N) \\ &\ge 2N \log N + \varepsilon a_{K+1}^{1/2} N - O(N \log \log N) \\ &\ge 2N \log N + \frac{\varepsilon}{2} a_{K+1}^{1/2} N \\ &\ge 2N \log N + N \varphi (\log N)^{1/2} . \end{split} \]
In particular, the upper asymptotic density of the set
\[ \Bigg\{ N \in \mathbb{N} \, : \, \sum_{\substack{n=1 \\ \| n \alpha \| \ge c/N}}^N \frac{1}{\| n \alpha \|} \ge 2N \log N + N \varphi(\log N)^{1/2}  \Bigg\} \]
is at least $1-4c^{1/2} \varepsilon$. Since $\varepsilon>0$ was arbitrary, the upper asymptotic density is 1, as claimed.

Similarly, for a.e.\ $\alpha$ we have $a_{K+1}=\max_{1 \le k \le K+1} a_k \ge 2 \varphi (100K)+K \log K \log \log K$ for infinitely many $K$. Theorem \ref{secondsumtheorem} gives that for infinitely many $K$ and any $q_K \le N < q_{K+1}$,
\[ \begin{split} \sum_{n=1}^N \frac{1}{n \left\| n \alpha \right\|} &\ge (\log N)^2 + a_{K+1} - O (a_{K+1}^{1/2} \log a_{K+1} + \log N \log \log N) \\ &\ge (\log N)^2 + \frac{1}{2} a_{K+1} \ge (\log N)^2 + \varphi (\log N) . \end{split} \]
In particular, the set
\[ \Bigg\{ N \in \mathbb{N} \, : \, \sum_{n=1}^N \frac{1}{n \| n \alpha \|} \ge (\log N)^2 + \varphi (\log N) \Bigg\} \]
has upper asymptotic density 1, as claimed.
\end{proof}

\subsection{Badly approximable irrationals}

In the special case of a badly approximable $\alpha$ Theorems \ref{firstsumtheorem} and \ref{secondsumtheorem} yield the value of the sums in \eqref{diophantinesums} up to an error $O(N \log \log N)$ resp.\ $O(\log N \log \log N)$. We now show how to modify the proof to remove the factor $\log \log N$ from the error terms.

\begin{proof}[Proof of Theorem \ref{badlyapproximabletheorem}] We have $\| n \alpha \| \ge c/N$ for all $1 \le n \le N$ with a suitably small constant $0<c<1/2$ depending only on $\alpha$, hence
\[ \sum_{n=1}^N \frac{1}{\| n \alpha \|} = \int_{c/N}^{\infty} \frac{1}{t^2} |\{ 1 \le n \le N \, : \, \| n \alpha \| \le t \}| \, \mathrm{d}t . \]
The contribution of the integral on $[1/2,\infty)$ is negligible:
\[ \int_{1/2}^{\infty} \frac{1}{t^2} |\{ 1 \le n \le N \, : \, \| n \alpha \| \le t \}| \, \mathrm{d}t \le \int_{1/2}^{\infty} \frac{N}{t^2} \, \mathrm{d}t = 2N . \]
To estimate the integral on $[c/N,1/2]$, we use the local discrepancy estimates \cite[Theorem 2]{SCH}
\[ \begin{split} \max_{1 \le N<q_{K+1}} \left( |\{ 1 \le n \le N \, : \, \{ n \alpha \} \le t \}| - tN \right) &= \sum_{\substack{k=1 \\ k \textrm{ even}}}^K \{ q_k t \} \left( a_{k+1} (1-\{ q_k t \}) + \{ q_{k+1}t \} - \{ q_{k-1} t \} \right) + O(1),  \\ \min_{1 \le N<q_{K+1}} \left( |\{ 1 \le n \le N \, : \, \{ n \alpha \} \le t \}| - tN \right) &= - \sum_{\substack{k=1 \\ k \textrm{ odd}}}^K \{ q_k t \} \left( a_{k+1} (1-\{ q_k t \}) + \{ q_{k+1}t \} - \{ q_{k-1} t \} \right) + O(1)  \end{split} \]
with universal implied constants, where $\{ \cdot \}$ is the fractional part function. In the special case of a badly approximable $\alpha$ these estimates immediately show that for all $1 \le N < q_{K+1}$,
\[ \left| |\{ 1 \le n \le N \, : \, \{ n \alpha \} \le t \}| - tN \right| \ll \sum_{k=1}^K \{ q_k t \} +1 \ll \sum_{\substack{k=1 \\ q_k \le 1/t}}^K q_k t + \sum_{\substack{k=1 \\ q_k > 1/t}}^K 1 +1 \ll \log (tN)+1. \]
In the last step we used the fact that $\sum_{j=1}^k q_j \le 3q_k$, and that there are $\ll \log (B/A)+1$ convergent denominators $q_k$ that fall in any given interval $[A,B]$. Since $-\alpha$ is also badly approximable, we can similarly estimate the number of $1 \le n \le N$ such that $\{ - n \alpha \} \le t$, that is, $\{ n \alpha \} \ge 1-t$. In particular,
\begin{equation}\label{nonuniform}
|\{ 1 \le n \le N \, : \, \| n \alpha \| \le t \}| = 2tN + O \left( \log (tN) +1 \right) \quad \textrm{uniformly in } 0<t \le 1/2 .
\end{equation}
We mention that \eqref{nonuniform} can also be easily deduced from an explicit formula for the local discrepancy due to T.\ S\'os \cite{SO}. Hence
\[ \int_{c/N}^{1/2} \frac{1}{t^2} |\{ 1 \le n \le N \, : \, \| n \alpha \| \le t \}| \, \mathrm{d}t = 2N \log N +O(N), \]
and the claim $\sum_{n=1}^N 1/ \| n \alpha \| = 2N \log N +O(N)$ follows. Summation by parts then yields $\sum_{n=1}^N 1/(n \| n \alpha \|) = (\log N)^2 + O(\log N)$, as claimed.
\end{proof}

\section{Related Diophantine sums}\label{relatedsection}

\subsection{Sums with fractional parts}

Theorems \ref{badlyapproximabletheorem}, \ref{aetheorem}, \ref{firstsumtheorem}, \ref{secondsumtheorem} have perfect analogues with the distance from the nearest integer function $\| \cdot \|$ replaced by the fractional part function $\{ \cdot \}$.
\begin{thm}\label{fractionalpartfirstsumtheorem} Let $c>0$ be an arbitrary constant. For any $K \ge 0$ and $q_K \le N<q_{K+1}$,
\[ \sum_{\substack{n=1 \\ \{ n \alpha \} \ge c/N}}^N \frac{1}{\{ n \alpha \}} = N \log N + O \left( (\mathds{1}_{\{ K+1 \textrm{ odd} \}}a_{K+1}^{1/2} + \log s_{K+1})N \right) . \]
If in addition $4 (ca_{K+1})^{1/2} q_K \le N$ and $K+1$ is odd, then
\[ \sum_{\substack{n=1 \\ \{ n \alpha \} \ge c/N}}^N \frac{1}{\{ n \alpha \}} \ge N \log N + q_{K+1} - O((\log s_{K+1})N) . \]
The implied constants depend only on $c$. Further, for any $K \ge 0$ and $q_K \le N < q_{K+1}$,
\[ \begin{split} \sum_{n=1}^N \frac{1}{n \{ n \alpha \}} = &\frac{1}{2} (\log N)^2 + \frac{\pi^2}{6} \sum_{\substack{k=1 \\ k \textrm{ odd}}}^K a_k + \mathds{1}_{\{ K+1 \textrm{ odd} \}} a_{K+1} \sum_{1 \le j \le N/q_K} \frac{1}{j^2} \\ &+O \Bigg( \sum_{\substack{k=1 \\ k \textrm{ odd}}}^{K+1} a_k^{1/2} \log a_k + (\log s_{K+1}) \log N \Bigg) \end{split} \]
with a universal implied constant.
\end{thm}

\begin{proof} This is a straightforward modification of the proof of Theorems \ref{firstsumtheorem} and \ref{secondsumtheorem}. The parity conditions follow from the fact that $p_{2k}/q_{2k} < \alpha < p_{2k+1}/q_{2k+1}$ for all $k \ge 0$.
\end{proof}

The same holds for the sums
\[ \sum_{\substack{n=1 \\ 1-\{ n \alpha \} \ge c/N}}^N \frac{1}{1-\{ n \alpha \}} \qquad \textrm{and} \qquad \sum_{n=1}^N \frac{1}{n (1-\{ n \alpha \})} \]
with ``odd'' replaced by ``even''. Theorem \ref{aetheorem} also remains true with $\| n \alpha \|$ replaced either by $\{ n \alpha \} /2$ or by $(1-\{ n \alpha \})/2$. The proof is identical to that of Theorem \ref{aetheorem}; note that the Borel--Bernstein theorem holds without any monotonicity assumption on $\varphi$, so the parity conditions do not cause any difficulty.

A straightforward modification of the proof of Theorem \ref{badlyapproximabletheorem} similarly shows that for any badly approximable $\alpha$,
\[ \sum_{n=1}^N \frac{1}{\{ n \alpha \}} = N \log N + O(N) \qquad \textrm{and} \qquad \sum_{n=1}^N \frac{1}{n \{ n \alpha \}} = \frac{1}{2} (\log N)^2 + O(\log N), \]
and the same holds with $\{ n \alpha \}$ replaced by $1-\{ n \alpha \}$.

\subsection{Shifted sums}

Some of our methods apply to shifted Diophantine sums as well. Here we only focus on the case of a badly approximable $\alpha$, for which we find the precise asymptotics. All previously known results on shifted Diophantine sums have a constant factor gap between the upper and the lower estimates \cite{BHV,BL}.
\begin{thm}\label{shiftedtheorem} Let $\alpha$ be a badly approximable irrational, and let $\beta \in \mathbb{R}$. For any $N \ge 1$,
\[ \sum_{\substack{n=1 \\ n \neq n_0}}^N \frac{1}{\| n \alpha + \beta \|} = 2N \log N + O(N \log \log N) \]
with an implied constant depending only on $\alpha$, where $n_0=n_0(\alpha, \beta, N) \in [1,N]$ is an integer for which $\min_{1 \le n \le N} \| n \alpha + \beta \| = \| n_0 \alpha + \beta \|$. If in addition $\inf_{n \in \mathbb{N}} (n \log \log n) \| n \alpha + \beta \| >0$, then
\[ \sum_{n=1}^N \frac{1}{\| n \alpha + \beta \|} = 2N \log N + O(N \log \log N), \qquad \sum_{n=1}^N \frac{1}{n \| n \alpha + \beta \|} = (\log N)^2 + O(\log N \log \log N) \]
with implied constants depending only on $\alpha$ and $\beta$.
\end{thm}

\begin{proof} For any integers $1 \le n < n' \le N$, we have $\| (n \alpha + \beta) - (n' \alpha + \beta) \| \ge 2c/N$ with a suitably small constant $c>0$ depending only on $\alpha$. In particular, for all $1 \le n \le N$, $n \neq n_0$ we have $\| n \alpha + \beta \| \ge c/N$. One readily checks that formulas \eqref{integral}, \eqref{integral1/2infty}, \eqref{integralR/N1/2} and \eqref{pigeonhole} in the proof of Theorem \ref{firstsumtheorem} remain true with an arbitrary shift $\beta$, hence
\[ \sum_{\substack{n=1 \\ n \neq n_0}}^N \frac{1}{\| n \alpha + \beta \|} = 2N \log N + O(N \log \log N) , \]
as claimed.

Under the additional assumption $\inf_{n \in \mathbb{N}} (n \log \log n) \| n \alpha + \beta \| >0$ the contribution of the $n=n_0$ term is $O(N \log \log N)$, thus $\sum_{n=1}^N 1/\| n \alpha + \beta \| = 2N \log N + O(N \log \log N)$, as claimed. Summation by parts then yields $\sum_{n=1}^N 1/(n \| n \alpha + \beta \|) = (\log N)^2 + O(\log N \log \log N)$, as claimed.
\end{proof}

Similarly, for any badly approximable $\alpha$ and any $\beta \in \mathbb{R}$,
\[ \sum_{\substack{n=1 \\ n \neq n_0'}}^N \frac{1}{\{ n \alpha + \beta \}} = N \log N + O(N \log \log N) \]
with an implied constant depending only on $\alpha$, where $n_0' = n_0'(\alpha, \beta, N) \in [1,N]$ is an integer for which $\min_{1 \le n \le N} \{ n \alpha + \beta \} = \{ n_0' \alpha + \beta \}$. If in addition $\inf_{n \in \mathbb{N}} (n \log \log n) \{ n \alpha + \beta \} >0$, then
\[ \sum_{n=1}^N \frac{1}{\{ n \alpha + \beta \}} = N \log N + O(N \log \log N), \qquad \sum_{n=1}^N \frac{1}{n \{ n \alpha + \beta \}} = \frac{1}{2} (\log N)^2 + O(\log N \log \log N) \]
with implied constants depending only on $\alpha$ and $\beta$. The same holds with $\{ n \alpha + \beta \}$ replaced by $1-\{ n \alpha + \beta \}$.

\subsection{A higher dimensional generalization}

There are several natural higher dimensional generalizations of the sums in \eqref{diophantinesums}, which have been studied using Fourier analysis \cite{BE1,LV}, geometry of numbers \cite{BHV} and lattices \cite{FR1,FR2}. In this setting, a vector $\alpha \in \mathbb{R}^d$ is called badly approximable if $\inf_{n \in \mathbb{Z}^d \backslash \{ 0 \}} \| n \|_{\infty}^d \cdot \| n_1 \alpha_1 + \cdots +n_d \alpha_d  \| >0$, where $\| n \|_{\infty} = \max_{1 \le k \le d} |n_k|$. Here we only comment on a result of Fregoli \cite{FR1}, who showed that for any badly approximable vector $\alpha \in \mathbb{R}^d$,
\[ N^d \log N \ll \sum_{n \in [-N,N]^d \backslash \{ 0\} } \frac{1}{\| n_1 \alpha_1 + \cdots + n_d \alpha_d \|} \ll N^d \log N . \]
His main result in fact yields the precise asymptotics of the previous sum, in particular giving an alternative proof of Theorem \ref{badlyapproximabletheorem} based on lattices.

\begin{thm} Let $\alpha \in \mathbb{R}^d$ be a badly approximable vector. For any $N \ge 1$,
\[ \begin{split} \sum_{n \in [-N,N]^d \backslash \{ 0\} } \frac{1}{\| n_1 \alpha_1 + \cdots + n_d \alpha_d \|} &= d 2^{d+1} N^d \log N + O(N^d), \\ \sum_{n \in [-N,N]^d \backslash \{ 0\} } \frac{1}{\| n \|_{\infty}^d \cdot \| n_1 \alpha_1 + \cdots + n_d \alpha_d \|} &= d^2 2^d (\log N)^2 + O(\log N) \end{split} \]
with implied constants depending only on $\alpha$.
\end{thm}

\begin{proof} The main result \cite[Proposition 1.3]{FR1} states that
\[ |\{ n \in [-N,N]^d \backslash \{ 0 \} \, : \, \| n_1 \alpha_1 + \cdots + n_d \alpha_d \| \le t \}| = 2^{d+1} t N^d + O \left( t^{d/(d+1)} N^{d^2/(d+1)} \right) \]
uniformly in $0<t \le 1/2$. Using this instead of formula \eqref{nonuniform}, the rest of the proof is identical to that of Theorem \ref{badlyapproximabletheorem}. Writing
\[ \sum_{n \in [-N,N]^d \backslash \{ 0\} } \frac{1}{\| n \|_{\infty}^d \cdot \| n_1 \alpha_1 + \cdots + n_d \alpha_d \|} = \sum_{\ell=1}^N \frac{1}{\ell^d} \sum_{\| n \|_{\infty}=\ell} \frac{1}{\| n_1 \alpha_1 + \cdots + n_d \alpha_d \|}, \]
the second claim follows from summation by parts.
\end{proof}

\section*{Acknowledgments} The author is supported by the Austrian Science Fund (FWF) project M 3260-N.

\end{document}